\documentclass[11pt]{article}
\usepackage{booktabs} 
\usepackage{fullpage, times}

\usepackage{times}
\usepackage{graphicx} 
\usepackage{subfigure} 

\usepackage{algorithm}
\usepackage{algorithmic}

\usepackage{hyperref}

\usepackage{amsmath}
\usepackage{color}
\usepackage{amsfonts}
\usepackage{amsthm}
\usepackage{hyperref}
\usepackage{amssymb}
\usepackage{bbm}
\usepackage{graphicx}

\usepackage{algorithm}
\usepackage{algorithmic}

\usepackage{mathtools}
\usepackage{hyperref}

\usepackage{enumitem}

\usepackage{verbatim}

\newtheorem{definition}{Definition}
\newtheorem{lemma}{Lemma}
\newtheorem{corollary}{Corollary}
\newtheorem{theorem}{Theorem}

\newtheorem*{conjecture*}{Conjecture}

\hypersetup{linktocpage}

{
\vspace{0.01cm}
\begin{center}
	\begin{boxedminipage}{0.8\linewidth}
		\begin{center}
			\textbf{\texttt{#1}}
		\end{center}
	} %
	{
	\end{boxedminipage}
\end{center}
 \vspace{0.3cm}
}

\newcommand{\bx}{{\mathbf x}}
\newcommand{\bc}{{\mathbf c}}
\newcommand{\bq}{{\mathbf q}}

\newcommand{\by}{{\mathbf y}}

\newcommand{\bv}{{\mathbf v}}

\newcommand{\be}{{\mathbf e}}

\newcommand{\bE}{\mathbb{E}}

\newcommand{\reals}{\mathbb{R}}

\newcommand{\cC}{\mathcal{C}}

\newcommand{\cF}{\mathcal{F}}

\newcommand{\cO}{\mathcal{O}}
\newcommand{\cP}{\mathcal{P}}

\newcommand{\bones}{\mathbbm{1}}
\renewenvironment{proof}{\par\noindent{\bf Proof\ }}{\hfill\BlackBox\\[2mm]}
\newcommand{\BlackBox}{\rule{1.5ex}{1.5ex}}

\newcommand{\RNum}[1]{\uppercase\expandafter{\romannumeral #1\relax}}
\makeatletter
\def\moverlay{\mathpalette\mov@rlay}
\def\mov@rlay#1#2{\leavevmode\vtop{%
   \baselineskip\z@skip \lineskiplimit-\maxdimen
   \ialign{\hfil$\m@th#1##$\hfil\cr#2\crcr}}}
\newcommand{\charfusion}[3][\mathord]{
    #1{\ifx#1\mathop\vphantom{#2}\fi
        \mathpalette\mov@rlay{#2\cr#3}
      }
    \ifx#1\mathop\expandafter\displaylimits\fi}
\makeatother

\DeclareMathOperator*{\argmin}{argmin} 
\DeclareMathOperator*{\argmax}{argmax} 

\renewcommand{\eqref}[1]{Equation~(\ref{#1})}
\newcommand{\ineqref}[1]{Inequality~(\ref{#1})}

\newcommand{\thmref}[1]{Theorem~\ref{#1}}
\newcommand{\lemref}[1]{Lemma~\ref{#1}}
\newcommand{\defref}[1]{Definition~\ref{#1}}

\newcommand{\norm}[1]{\left\Vert#1\right\Vert}
\newcommand{\normsq}[1]{\left\Vert#1\right\Vert^2}
\newcommand{\inprod}[2]{ \left< #1 , #2 \right>}
\newcommand{\circpar}[1]{\left( #1 \right)}

\newcommand{\dom}{\text{dom}}

\newcommand{\diag}[1]{\text{diag}\circpar{#1}}

\newcommand{\bigO}[1]{\mathcal{O}{\left(#1\right)}}
\newcommand{\bigtO}[1]{\tilde{\mathcal{O}}{\left(#1\right)}}

\begin{document} 

%
%
%

\title{On the Iteration Complexity of Oblivious First-Order \\Optimization Algorithms}
\author{Yossi Arjevani\\Weizmann Institute of Science\\\texttt{yossi.arjevani@weizmann.ac.il}
\and
 Ohad Shamir\\Weizmann Institute of Science\\{\texttt{ohad.shamir@weizmann.ac.il}}}
\date{}
\maketitle
																				\begin{abstract}

We consider a broad class of first-order optimization algorithms which are \emph{oblivious}, in the sense that their step sizes are scheduled regardless of the function under consideration, except for limited side-information such as smoothness or strong convexity parameters. 
With the knowledge of these two parameters, we show that any such algorithm attains an iteration complexity lower bound of $\Omega(\sqrt{L/\epsilon})$ for $L$-smooth convex functions, and  $\tilde{\Omega}(\sqrt{L/\mu}\ln(1/\epsilon))$ for $L$-smooth $\mu$-strongly convex functions. These lower bounds are stronger than those in the traditional oracle model, as they hold independently of the dimension. To attain these, we abandon the oracle model in favor of a structure-based approach which builds upon a framework recently proposed in \cite{arjevani2015lower}. We further show that without knowing the strong convexity parameter, it is impossible to attain an iteration complexity better than  $\tilde{\Omega}\circpar{(L/\mu)\ln(1/\epsilon)}$. This result is then used to formalize an observation regarding $L$-smooth convex functions, namely, that the iteration complexity of algorithms employing time-invariant step sizes must be at least $\Omega(L/\epsilon)$.

																					\end{abstract}

																				\section{Introduction}
The ever-increasing utility of mathematical optimization in machine learning and other fields has led to a great interest in understanding the computational boundaries of solving optimization problems. Of a particular interest is the class of unconstrained smooth, and possibly strongly convex, optimization problems. Formally, we consider the following problem,
\begin{align*} 
	\min_{\bx\in\reals^d} f(\bx) 
\end{align*}
where $f:\reals^d\to\reals$ is convex and $L$-smooth, i.e.,
\begin{align*}
	\norm{\nabla f(\bx)- \nabla f(\by)} &\le L\norm{\bx-\by}
\end{align*}
for some $L>0$, and possibly \emph{$\mu$-strongly convex}, that is,
\begin{align*}
		f(\by) \ge f(\bx) + \inprod{\by-\bx}{\nabla f(\bx)} + \frac{\mu}{2}\normsq{\by-\bx}
\end{align*}
for some $\mu>0$. In this work, we address the question as to how fast can one expect to solve this sort of problems to a prescribed level of accuracy, using methods which are based on first-order information (gradients, or more generally sub-gradients) alone.

The standard approach to quantify the computational hardness of optimization problems is through the \emph{Oracle Model}. In this approach, one models the interaction of a given optimization algorithm with some instance from a class of functions as a sequence of queries, issued by the algorithm, to an external first-order oracle procedure. Upon receiving a query point $\bx\in\reals^d$, the oracle reports the corresponding value $f(\bx)$ and gradient $\nabla f(\bx)$.  In their seminal work, Nemirovsky and Yudin \cite{nemirovskyproblem} showed that for any first-order optimization algorithm, there exists an $L$-smooth and $\mu$-strongly convex function $f:\reals^d\to\reals$ such that the number of queries required to obtain an \emph{$\epsilon$-optimal} solution $\tilde{\bx}$ which satisfies
$$f(\tilde{\bx}) < \min_{\bx\in\reals^d} f(\bx)+ \epsilon,$$
is at least\footnote{Following standard conventions, here, tilde notation hides logarithmic factors in the smoothness parameter, the strong convexity parameter and the distance of the initialization point from the minimizer.}
\begin{align} \label{ineq:sqrtlb_lb}
	&\tilde{\Omega}\circpar{\min\left\{d, \sqrt{\kappa}\right\} \ln(1/\epsilon)},&\mu>0\\
	&\tilde{\Omega}{(\min\{d\ln(1/\epsilon),\sqrt{L/\epsilon}\})},&\mu=0\nonumber
\end{align}
where $\kappa\coloneqq L/\mu$ is the so-called \emph{condition number}. This lower bound, although based on information considerations alone, is tight. Concretely, it is achieved by a combination of Nesterov's well-known accelerated gradient descent (AGD, \cite{nesterov1983method}) with an iteration complexity of 
\begin{align} \label{ineq:agd_ub}
	&\bigtO{\sqrt{\kappa} \ln(1/\epsilon)},&\mu>0\\
	&\bigO{\sqrt{L/\epsilon}},&\mu=0,\nonumber
\end{align}
and the center of gravity method (MCG, \cite{levin1965algorithm,newman1965location}) whose iteration complexity is 
\begin{align*}
	\bigO{d\ln(1/\epsilon)}. 
\end{align*}

Although the combination of MCG and AGD appear to achieve optimal iteration complexity, this is not the case when focusing on \emph{computationally efficient} algorithms. In particular, the per-iteration cost of MCG scales poorly with the problem dimension, rendering it impractical for high-dimensional problems. In other words, not taking into account the computational resources needed for processing first-order information limits the ability of the oracle model to give a faithful picture of the complexity of optimization.\\

To overcome this issue \cite{arjevani2015lower} recently proposed the framework of $p$-Stationary Canonical Linear Iterative ($p$-SCLI) in which, instead of modeling the way algorithms acquire information on the function at hand, one assumes certain dynamics which restricts the way new iterates are being generated. This framework includes a large family of computationally efficient first-order algorithms, whose update rule, when applied on quadratic functions, reduce to a recursive application of some fixed linear transformation on the most recent $p$ points (in other words, $p$ indicates the number of previous iterates stored by the algorithm in order to compute a new iterate). The paper showed that the iteration complexity of $p$-SCLIs over smooth and strongly convex functions is bounded from below by 
\begin{align} \label{ineq:pscli_lb}
	\tilde{\Omega}\circpar{ \sqrt[p]{\kappa} \ln(1/\epsilon)}.
\end{align}

Crucially, as opposed to the classical lower bounds in (\ref{ineq:sqrtlb_lb}), the lower bound in (\ref{ineq:pscli_lb}) holds for any dimension $d>1$. This implies that even for fixed $d$, the iteration complexity of $p$-SCLI algorithms must scale with the condition number. That being said, the lower bound in (\ref{ineq:pscli_lb}) raises a few major issues which we wish to address in this work:
\begin{itemize}[leftmargin=*]
\item  Practical first-order algorithms in the literature only attain this bound for $p=1,2$ (by standard gradient descent and AGD, respectively), so the lower bound appears intuitively loose. Nevertheless, \cite{arjevani2015lower} showed that this bound is actually tight for all $p$. The reason for this discrepancy is that the bound for $p>2$ was shown to be attained by $p$-SCLI algorithms whose updates require exact knowledge of spectral properties of the Hessian, which is computationally prohibitive to obtain in large-scale problems. In this work, we circumvent this issue by systematically considering the \emph{side-information} available to the algorithm. In particular, we show that under the realistic assumption, that the algorithm may only utilize the strong convexity and smoothness of the objective function, the lower bound in (\ref{ineq:pscli_lb}) can be substantially improved. 

	\item The lower bound stated above is limited to \emph{stationary} optimization algorithms whose coefficients $\alpha_j,\beta_j$ are not allowed to change in time (see Section \ref{subsection:pcli_classification}). 

\item The formulation suggested in \cite{arjevani2015lower} does not allow generating more than one iterate at a time. This requirement is not met by many popular optimization problems for finite sums minimization.

\item Lastly, whereas the proofs in \cite{arjevani2015lower} are elaborate and technically complex, the proofs we provide here are relatively short and simple.
\end{itemize}

In its simplest form, the framework we consider is concerned with algorithms which generate iterates by applying the following simple update rule repeatedly: 
\begin{align} \label{def:basic_form}
	\bx^{(k+1)} &=\sum_{j=1}^{p} \alpha_j \nabla f(\bx^{(k+1-j)}) + \beta_j \bx^{(k+1-j)},
\end{align} 
where $\alpha_j,\beta_j\in\reals$ denote the corresponding coefficients. 
A clear advantage of this class of algorithms is that, given the corresponding gradients, the computational cost of executing each update rule scales linearly with the dimension of the problem and $p$.

This basic formulation already subsumes popular first-order optimization algorithms. For example, at each iteration the Gradient Descent (GD) method generates a new iterate by computing a linear combination of the current iterate and the gradient of the current iterate, i.e., 
\begin{align}
	 \bx^{(k+1)}&= \bx^{(k)} + \alpha\nabla f(\bx^{(k)})
\end{align}
for some real scalar $\alpha$. Another important example is a stationary variant of AGD \cite{nesterov2004introductory} and the heavy-ball method (e.g., \cite{polyak1987introduction}) which generates iterates according to
\begin{align}\label{update:AGD}
	\bx^{(k+1)}&= \beta_1 \bx^{(k)} + \alpha_1 \nabla f(\bx^{(k)})\nonumber\\&~~+ \beta_2 \bx^{(k-1)} + \alpha_2 \nabla f(\bx^{(k-1)}).
\end{align}

In this paper, we follow a generalized form of (\ref{def:basic_form}) which is exhibited by standard optimization algorithms: GD, conjugate gradient descent, sub-gradient descent, AGD, the heavy-ball method, coordinate descent, quasi-Newton methods, ellipsoid method, etc. The main difference being how much effort one is willing to put in computing the coefficients of the optimization process. We call these methods first-order $p$-Canonical Linear Iterative optimization algorithms (in this paper, abbr. $p$-CLI). We note that our framework (as a method to prove lower bounds) also applies to stochastic algorithms, as long as the expected update rule (conditioned on the history) follows a generalized form similar to (\ref{def:basic_form}).

In the context of machine learning, many algorithms for minimizing finite sums of functions with, possibly, a regularization term (also known as, Regularized Empirical Risk Minimization) also fall into our framework, e.g., Stochastic Average Gradient (SAG, \cite{schmidt2013minimizing}), Stochastic Variance Reduction Gradient (SVRG, \cite{johnson2013accelerating}), Stochastic Dual Coordinate Ascent (SDCA, \cite{shalev2013stochastic}), Stochastic Dual Coordinate Ascent without Duality (SDCA without duality, \cite{shalev2015sdca}) and SAGA \cite{defazio2014saga}, to name a few, and as such,  are subject to the same lower bounds established through this framework.

In its full generality, the formulation of this framework is too rich to say much. In what follows, we shall focus on \emph{oblivious} p-CLIs, which satisfy the realistic assumption that the coefficients $\alpha_j,\beta_j$ do not depend on the specific function under consideration. Instead, they can only depend on time and some limited side-information on the function (this term will be made more precise in \defref{definition:side_information}). In particular, we show that the iteration complexity of oblivious $p$-CLIs over $L$-smooth and $\mu$-strongly convex functions whose coefficients are allowed to depend on $\mu$ and $L$ is 
\begin{align}\label{ineq:contribution_lower_bounds}
	&\tilde{\Omega}\circpar{\sqrt{\kappa} \ln(1/\epsilon)},&\mu>0\\
	&\tilde{\Omega}{(\sqrt{L/\epsilon})},&\mu=0\nonumber.
\end{align}
Note that, in addition to being \emph{dimension-independent} (similarly to (\ref{ineq:pscli_lb})), this lower bound holds regardless of $p$. We further stress that the algorithms discussed earlier which attain the lower bound stated in (\ref{ineq:pscli_lb}) are not oblivious and require more knowledge of the objective function. 

In the paper, we also demonstrate other cases where the side-information available to the algorithm crucially affects its performance, such as knowing vs. not knowing the strong convexity parameter.


Finally, we remark that this approach of modeling the structure of optimization algorithms, as opposed to the more traditional oracle model, can be also found in \cite{polyak1987introduction,lessard2014analysis,flammarion2015averaging,drori2014contributions}. However, whereas these works are concerned with upper bounds on the iteration complexity, in this paper we primarily focus on lower bounds.

To summarize, our main contributions are the following:
\begin{itemize}[leftmargin=*]
	\item In Section \ref{subsection:definitions}, we propose a novel framework which substantially generalizes the framework introduced in \cite{arjevani2015lower}, and includes a large part of modern first-order optimization algorithms. 
	
	\item In Section \ref{subsection:pcli_classification}, we identify within this framework the class of oblivious optimization algorithms, whose step sizes are scheduled regardless of the function at hand, and provide an iteration complexity lower bound as given in (\ref{ineq:contribution_lower_bounds}). We improve upon \cite{arjevani2015lower} by establishing lower bounds which hold both for smooth functions and smooth and strongly convex functions, using simpler and shorter proofs. Moreover, in addition to being \emph{dimension-independent}, the lower bounds we derive here are tight. In the context of machine learning optimization problems, the same lower bound is shown to hold on the bias of methods for finite sums with a regularization term, such as: SAG, SAGA, SDCA without duality and SVRG.
	
	\item Some oblivious algorithms for $L$-smooth and $\mu$-strongly convex functions admit a linear convergence rate using step sizes which are scheduled regardless of the strong convexity parameter (e.g., standard GD with a step size of $1/L$. See Section 3 in \cite{schmidt2013minimizing} and Section 5 in \cite{defazio2014saga}). In Section \ref{subsection:no_strong}, we show that adapting to 'hidden' strong convexity, without explicitly incorporating the strong convexity parameter, results in an inferior iteration complexity of
	\begin{align} \label{contribution:sc_lower_bound}
		\tilde{\Omega}\circpar{\kappa\ln(1/\epsilon)}.
	\end{align} 
	This result sheds some light on a major issue regarding scheduling step sizes of optimization algorithms .
	
	\item In Section \ref{subsection:acceleration_and_stationarity}, we discuss the class of stationary optimization algorithms, which use time-invariant step sizes, over $L$-smooth functions and show that they admit a tight iteration complexity of 
	\begin{align} \label{contribution:smooth_lower_bound}
		\Omega(L/\epsilon).
	\end{align}
	In particular, this bound implies that in terms of dependency on the accuracy parameter $\epsilon$, SAG and SAGA admit an optimal iteration complexity w.r.t. the class of stochastic stationary $p$-CLIs. 
	Acceleration schemes, such as \cite{frostig2015regularizing,lin2015universal}, are able to break this bound by re-scheduling these algorithms in a non-stationary (though oblivious) way.
\end{itemize}

																							\section{Framework} \label{section:framework}

\subsection{Definitions} \label{subsection:definitions}
In the sequel we present our framework for analyzing first-order optimization algorithms. We begin by providing a precise definition of a class of optimization problems, accompanied by some side-information. We then formally define the framework of $p$-CLI algorithms and the corresponding iteration complexity.

\begin{definition}[Class of Optimization Problems] \label{definition:side_information}
A class of optimization problems $\cC$ is an ordered pair of $(\cF,I)$, where $\cF$ is a family of functions which defined over the same domain, and $I:\cF\to\mathfrak{I}$ is a mapping which provides for each $f\in\cF$ the corresponding side-information element in some set $\mathfrak{I}$. The domain of the functions in $\cF$ is denoted by $\dom(\cC)$.
\end{definition}

For example, let us consider quadratic functions of the form $$\bx\mapsto\frac{1}{2}\bx^\top Q \bx + \bq^\top \bx,$$ where $Q\in\reals^{d\times d}$ is a positive semidefinite matrix whose spectrum lies in $\Sigma\subseteq\reals^{+}$, and $\bq\in\reals^d$. Here, each instance may be accompanied with either a complete specification of $\Sigma$; lower and upper bounds for $\Sigma$; just an upper bound for $\Sigma$; a rough approximation of $Q^{-1}$ (e.g., sketching techniques), etc. We will see that the exact nature of side-information strongly affects the iteration complexity, and that this differentiation between the family of functions under consideration and the type of side-information is not mere pedantry, but a crucial necessity. 

We now turn to rigorously define first-order $p$-CLI optimization algorithms. The basic formulation shown in (\ref{def:basic_form}) does not allow generating more than one iterate at a time. The framework which we present below relaxes this restriction to allow a greater generality which is crucial for incorporating optimization algorithms for finite sums (see Stochastic $p$-CLIs in Section \ref{subsection:pcli_classification}). We further extend (\ref{def:basic_form}) to allow non-differentiable functions and constraints into this framework, by generalizing gradients to sub-gradients. 

\theoremstyle{plain}
\begin{definition}\label{definition:pcli} [First-order $p$-CLI]
An optimization algorithm is called a first-order $p$-Canonical Linear Iterative ($p$-CLI) optimization algorithm over a class of optimization problems $\cC=(\cF,I(\cdot))$, if given an instance $f\in\cF$ and an arbitrary set of $p$ initialization points $\bx^0_{1},\dots,\bx^0_{p}\in\dom(\cC)$, it operates by iteratively generating points for which
\begin{align} \label{assumption:dynamics}
	\bx^{(k+1)}_i \in\sum_{j=1}^{p} \circpar{A^{(k)}_{ij} \partial f  + B^{(k)}_{ij}  }(\bx^{(k)}_{j}), \quad k=0,1,\dots 
\end{align} 
holds, where the coefficients $A^k_{ij},B^k_{ij}$ are some linear operators which may depend on $I(f)$. 
\end{definition}	
Formally, the expression $A^{(k)}_{ij} \partial f$ in (\ref{assumption:dynamics})  denotes the composition of  $A^{(k)}_{ij}$ and the sub-gradient operator. Likewise, the r.h.s. of (\ref{assumption:dynamics}) is to be understand as an evaluation of sum of two operators $A^{(k)}_{ij} \partial f$ and $B^{(k)}_{ij}$ at $\bx^{(k)}_{j}$. 

In this level of generality, this framework encompasses very different kinds of optimization algorithms. We shall see that various assumptions regarding the coefficients complexity and side-information yield different lower bound on the iteration complexity.

We note that although this framework concerns algorithms whose update rules are based on a fixed number of points, a large part of the results shown in this paper holds in the case where $p$ grows indefinitely in accordance with the number of iterations.

We now turn to provide a formal definition of \emph{Iteration Complexity}. We assume that the point returned after $k$ iterations is $\bx_p^{(k)}$. This assumption merely serves as a convention and is not necessary for our bounds to hold. 
\begin{definition}[Iteration Complexity]\label{definition:iteration_complexity}
The iteration complexity $\mathcal{IC}(\epsilon)$ of a given $p$-CLI w.r.t. a given problem class $\cC=(\cF,I)$ is defined to be the minimal number of iterations $K$ such that 
\begin{align*}
	  f(\bE\bx^{(k)}_p) - \min_{\bx\in\dom{\cC}} f(\bx) <\epsilon,\quad \forall k\ge K
\end{align*}
uniformly over $\cF$, where the expectation is taken over all the randomness introduced into the optimization process (see Stochastic $p$-CLIs below). 
\end{definition}
For simplicity, when stating bounds in this paper, we shall omit the dependency of the iteration complexity on the initialization points. The precise dependency can be found in the corresponding proofs.

													\subsection{\texorpdfstring{Classification of First-order $p$-CLIs and Scope of Work}{Classification of p-CLIs and Scope of Work}}

														\label{subsection:pcli_classification}
As mentioned before, we cannot say much about the framework in its full generality. In this paper, we restrict our attention to the following three (partially overlapping) classes of $p$-CLIs:
\begin{description}
\item[Stationary $p$-CLI] where the coefficients are allowed to depend exclusively on side-information (see \defref{definition:iteration_complexity}). In particular, the coefficients are not allowed to change with time. Seemingly restrictive, this class of $p$-CLIs subsumes many efficient optimization methods, especially when coupled with stochasticity (see below). Notable stationary $p$-CLIs are: GD with fixed step size \cite{nesterov2004introductory}, stationary AGD \cite{nesterov2004introductory} and the Heavy-Ball method \cite{polyak1987introduction}.

\item[Oblivious $p$-CLI] where the coefficients are allowed to depend on side-information, as well as to change in time. Notable algorithms here are GD and AGD with step sizes which are scheduled irrespectively of the function under consideration \cite{nesterov2004introductory} and the Sub-Gradient Descent method (e.g., \cite{shor2012minimization}). 


\item[Stochastic $p$-CLI] where (\ref{assumption:dynamics}) holds with respect to $\bE \bx^{(k)}_j$, that is, 
\begin{align} \label{assumption:stochastic_dynamics}
		\bE\bx^{(k+1)}_i \in\sum_{j=1}^{p} \circpar{A^{(k)}_{ij} \partial f  + B^{(k)}_{ij}  }(\bE\bx^{(k)}_{j}).
\end{align}
Stochasticity is an efficient machinery of tackling optimization problems where forming the gradient is prohibitive, but engineering an efficient unbiased estimator is possible. Such situations occur frequently in the context of machine learning, where one is interested in minimizing finite sums of large number of convex functions,
\begin{align*}
	\min_{\bx\in\reals^d} F(\bx)\coloneqq \sum_{i=1}^m f_i(\bx),
\end{align*}
in which case, forming a sub-gradient of $F$ at a given point may be too expensive. Notable optimization algorithms for variants of this problem are: SAG, SDCA without duality, SVRG and SAGA, all of which are stationary stochastic $p$-CLIs. Moreover, as opposed to algorithms which produce only one new point at each iteration (e.g., (\ref{def:basic_form})), these algorithms sometimes update a few points at the same time. To illustrate this, let us express SAG as a stochastic stationary $(m+1)$-CLI. In order to avoid the computationally demanding task of forming the exact gradient of $F$ at each iteration, SAG uses the first $m$ points to store estimates for the gradients of the individual functions $$\by_i\approx \nabla f_i(\bx^{(k)}_{m+1}),~i=1\dots m.$$ At each iteration, SAG sets $\by_i=\nabla f_i(\bx^{(k)}_{m+1})$ for some randomly chosen $i\in[m]$, and then updates $\bx^{(k)}_{m+1}$ accordingly, by making a gradient step with a fixed step size using the new estimate for $\nabla F(\bx^{(k)}_{m+1})$. This implies that the expected update rule of SAG is stationary and satisfies (\ref{assumption:stochastic_dynamics}).
\end{description}

As opposed to an oblivious schedule of step sizes, many optimization algorithms set the step sizes according to the  first-order information which is accumulated during the optimization process. A well-known example for such a non-oblivious schedule is conjugate gradient descent, whose update rule can be expressed as follows:
\begin{align}
\bx^{(k+1)}_1&=  \bx^{(k)}_2\nonumber\\
	\bx^{(k+1)}_2&=  (\alpha \partial f + (1+\beta)I ) \bx^{(k)}_2- \beta \bx^{(k)}_1,
\end{align}
where the step sizes are chosen so as to minimize $f(\bx^{(k+1)}_1)$ over $\alpha,\beta\in\reals$. Other algorithms employ coefficients whose schedule does not depend directly on first-order information. For example, at each iteration coordinate descent updates one coordinate of the current iterate, by completely minimizing the function at hand along some direction. In our formulation, such update rules are expressed using coefficients which are diagonal matrices. In a sense, the most expensive coefficients used in practice are the one employed by Newton method, which in this framework, may be expressed as follows:
\begin{align}
	\bx^{(k+1)}_1&= (I-\nabla^2(f)^{-1} \nabla f) \bx_1^{(k)} 
\end{align}

The algorithms mentioned above: conjugate gradient descent, coordinate descent and Newton methods; as well as other non-oblivious $p$-CLI optimization algorithms, such as quasi-Newton methods (e.g., \cite{nocedal2006numerical}) and the ellipsoid method (e.g., \cite{atallah1998algorithms}), will not be further considered in this paper.

																				\section{Lower Bounds on the Iteration Complexity of Oblivious \texorpdfstring{$p$}{p}-CLIs}\label{section:main_results}

Having formally defined the framework, we are now in position to state our first main result. Perhaps the most common side-information used by practical algorithms is the strong-convexity and smoothness parameters of the objective function. Oblivious $p$-CLIs with such side-information tend to have low per-iteration cost and a straightforward implementation. However, this lack of adaptivity to the function being optimized results in an inevitable lower bound on the iteration complexity:

\begin{theorem}\label{theorem:main_lower_bound}
Suppose the smoothness parameter $L$ and  the strong convexity parameter $\mu$ are known, i.e., $I(\cdot)=\{L,\mu\}$. Then the iteration complexity of any oblivious, possibly stochastic, $p$-CLI optimization algorithm is bounded from below by
\begin{align} \label{ineq:main_lower_bounds}
	&\tilde{\Omega}\circpar{\sqrt{\kappa}\ln(1/\epsilon)}, &\mu>0\\
	&\Omega{(\sqrt{L/\epsilon})},&\mu=0,\nonumber
\end{align}
where $\kappa\coloneqq L/\mu$ 
\end{theorem}
As discussed in the introduction, \thmref{theorem:main_lower_bound} significantly improves upon the lower obtained by \cite{arjevani2015lower} in 3 major aspects:
\begin{itemize}
	\item It holds for both smooth functions, as well as smooth and strongly convex functions.
	\item In both the strongly-convex and non-strongly convex cases, the bounds we derive are tight for $p>1$ (Note that if the coefficients are scalars and time-invariant, then for smooth and strongly convex functions a better lower bound of $\tilde{\Omega}(\kappa\ln(1/\epsilon))$ holds. See Theorem 8, \cite{arjevani2015lower}).
	\item It considers a much wider class of algorithms, namely, methods which may use different step size at each iteration and may freely update each of the $p$ points.
\end{itemize}
We stress again that, in contrast to (\ref{ineq:sqrtlb_lb}), this lower bound does not scale with the dimension of the problem.  

The proof of \thmref{theorem:main_lower_bound}, including logarithmic factors and constants which appear in the lower bound, is found in (\ref{proof:main_lower_bound}), and can be roughly sketched as follows. First, we consider $L$-smooth and $\mu$-strongly convex quadratic functions of the form
\begin{align*}
	\bx\mapsto \frac{1}{2} \bx^\top \diag{\eta}\bx + \eta \bones^\top \bx,\quad\eta\in[\mu,L],
\end{align*}
over $\reals^d$, all of which share the same minimizer, $$\bx^*= -\operatorname{diag}^{-1}(\eta) \eta \bones =-\bones.$$ Next, we observe that each iteration of $p$-CLI involves application of $A\partial f + B$, which is a linear expression in $\partial f$ whose coefficients are some linear operators, on the current points $\bx^{(k)}_j,~j=1,\dots,p$, which are then summed up to form the next iterate.  Applying this argument inductively, and setting the initialization points to be zero, we see that the point returned by the algorithm at the $k$'th iteration can be expressed as follows,
\begin{align*}
	\bx^{(k)}_p&= (s_1(\eta)\eta ,\dots,s_d(\eta)\eta)^\top,
\end{align*}
where $s_i(\eta)$ are real polynomials of degree $k-1$. Here, the fact that the coefficients are scheduled obliviously, i.e., \emph{do not} depend on the very choice of $\eta$, is crucial (when analyzing other types of $p$-CLIs, one may encounter cases where the coefficients of $s(\eta)$ are not constants, in which case the resulting expression may not be a polynomial). Bearing in mind that our goal is to bound the distance to the minimizer $-\bones$ (which, in thic case, is equivalent to the iteration complexity up to logarithmic factors), we are thus led to ask how small can $|s(\eta)\eta+1|$ be. Formally, we aim to bound 
\begin{align*}
	\max_{\eta\in[\mu,L]} |s(\eta)\eta+1|
\end{align*}
	from below. To this end, we use the properties of the well-known Chebyshev polynomials, by which we derive the following lower bound:
$$\min_{s(\eta)\in \reals[\eta], \partial(s) = k-1  } ~\max_{\eta\in[\mu,L]} |s(\eta)\eta+1|\ge\circpar{\frac{\sqrt\kappa-1}{\sqrt\kappa+1}}^{k}.$$
The proof of the smooth non-strongly convex case is also based on a reduction from a minimization problem to a polynomial approximation problem, only this time the resulting approximation problem is slightly different (see \eqref{quadratic_function_for_smooth_case} in Appendix \ref{proof:acceleration_lower_bound}).

The idea of reducing optimization bounds to polynomial approximation problems is not new, and is also found for instance in \cite{nemirovskyproblem}, where lower bounds under the oracle model are derived. In particular, both approaches, the oracle model and $p$-CLI, exploit the idea that when applied on some strongly convex quadratic functions $\frac{1}{2}\bx^\top Q\bx + \bq^\top\bx$ over $\reals^d$, the $k$'th iterate can be expressed as $s(Q)\bq$ for some real polynomial $s(\eta)\in\reals[\eta]$ of degree at most $k-1$. Bounding the iteration complexity is then essentially reduced to the question of how well can we approximate $Q^{-1}$ using such polynomials. However, the approach here uses a fundamentally different technique for achieving this, and whereas the oracle model does not impose any restrictions on the coefficients of $s(\eta)$, the framework of $p$-CLIs allows us to effectively control the way these coefficients are being produced. The excessive freedom in choosing $s(\eta)$ constitutes a major weakness in the oracle model and prevents obtaining iteration complexity bounds significantly larger than the dimension $d$. To see why, note that by the Cayley-Hamilton theorem, there exists a real polynomial $s(\eta)$ of degree at most $d-1$ such that $s(Q)=-Q^{-1}$. Therefore, the $d$'th iterate can potentially be $s(Q)\bq=-Q^{-1}\bq$, the exact minimizer. We avoid this limited applicability of the oracle model by adopting a more structural approach, which allows us to restrict the kind of polynomials which can be produced by practical optimization algorithms. Furthermore, our framework is more flexible in the sense that the coefficients of $s(\eta)$ may be formed by optimization algorithms which do not necessarily fall into the category of first-order algorithms, e.g., coordinate descent.

It is instructive to contrast our approach with another structural approach for deriving lower bounds which was proposed by \cite{nesterov2004introductory}. Nesterov \cite{nesterov2004introductory} considerably simplifies the technique employed by Nemirovsky and Yudin \cite{nemirovskyproblem} at the cost of introducing additional assumption regarding the way new iterates are generated. Specifically, it is assumed that each new iterate lies in the span of all the gradients acquired earlier. Similarly to \cite{nemirovskyproblem}, this approach also does not yield dimension-independent lower bounds. Moreover, such an approach may break in presence of conditioning mechanisms (which essentially, aim to handle poorly-conditioned functions by multiplying the corresponding gradients by some matrix). In our framework, such conditioning is handled through non-scalar coefficients. Thus, as long as the conditioning matrices depend solely on $\mu,L$ our lower bounds remain valid.

													\section{Side-Information in Oblivious Optimization} 
																	\label{section:side_information}
																
\subsection{No Strong Convexity Parameter, No Acceleration }\label{subsection:no_strong}

Below we discuss the effect of not knowing exactly the strong convexity parameter on the iteration complexity of oblivious $p$-CLIs. In particular, we show that the ability of oblivious $p$-CLIs to obtain iteration complexity which scales like $\sqrt{\kappa}$ crucially depends on the quality of the strong convexity estimate of the function under consideration. Moreover, we show that stationary $p$-CLIs are strictly weaker than general oblivious $p$-CLIs for smooth non-strongly convex functions, in the sense that stationary $p$-CLIs cannot obtain an iteration complexity of $\mathcal{O}(\sqrt{L/\epsilon})$.
%

The fact that decreasing the amount of side-information increases the iteration complexity is best demonstrated by a family of quadratic functions which we already discussed before, namely, $$\bx\mapsto\frac{1}{2}\bx^\top Q \bx + \bq^\top \bx,$$ where $Q\in\reals^{d\times d}$ is positive semidefinite whose spectrum lies in $\Sigma\subseteq\reals^+$ and $\bq\in\reals^d$. In Theorem 8 in \cite{arjevani2015lower}, it is shown that if $Q$ is given in advance, but $\bq$ is unknown, then the iteration complexity of stationary  $p$-CLIs which follows (\ref{def:basic_form}) is 
\begin{align*}
	\tilde{\Omega}(\sqrt[p]{\kappa}\ln(1/\epsilon)). 
\end{align*}
It is further shown that this lower bound is tight (see Appendix A in \cite{arjevani2015lower}). In \thmref{theorem:main_lower_bound} we show that if both the smoothness and the strong convexity parameters $\{\mu,L\}$ are known then the corresponding lower bound for this kind of algorithms is 
\begin{align*}
	\tilde{\Omega}(\sqrt{\kappa}\ln(1/\epsilon)).
\end{align*}
As mentioned earlier, this lower bound is tight and is attained by a stationary version of AGD.

However, what if only the smoothness parameter $L$ is known a-priori? The following theorem shows that in this case the iteration complexity is substantially worse. For reasons which will become clear later, it will be convenient to denote the strong convexity parameter and the condition number of a given function $f$ by $\mu(f)$ and $\kappa(f)$, respectively.
\begin{theorem}\label{theorem:acceleration_lower_bound}
Suppose that only $L$ the smoothness parameter is known, i.e. $I(\cdot)=\{L\}$. If the iteration complexity of a given oblivious, possibly stochastic, $p$-CLI optimization algorithm is 
\begin{align}
	\tilde{O}(\kappa(f)^\alpha\ln(1/\epsilon)),
\end{align}
then $\alpha\ge1$.  
\end{theorem}
\thmref{theorem:acceleration_lower_bound} pertains to the important issue of optimal schedules for step sizes. Concretely, it implies that, in the absence of the strong convexity parameter, one is still able to schedule the step sizes according to the smoothness parameter so as to obtain exponential convergence rate, but only to the limited extent of linear dependency on the condition number (as mentioned before, this sub-optimality in terms of dependence on the condition number, can be also found in \cite{schmidt2013minimizing} and \cite{defazio2014saga}). This bound is tight and is attained by standard gradient descent (GD). 

\thmref{theorem:acceleration_lower_bound} also emphasizes the superiority of standard GD in cases where the true strong convexity parameter is poorly estimated. Such situations may occur when one underestimate the true strong convexity parameter by following the strong convexity parameter introduced by an explicit regularization term. Specifically, if $\hat{\mu}$ denotes our estimate for 
the true strong convexity parameter $\mu$ (obviously, $\hat{\mu}<\mu$  to ensure convergence), then \thmref{theorem:main_lower_bound} already implies that, for a fixed accuracy level, the worst iteration complexity of our algorithm is on the order of $\sqrt{L/\hat{\mu}}$, whereas standard GD with $1/L$ step sizes has iteration complexity on the order of $L/\mu$. Thus, if our estimate is too conservative, i.e.,  $\hat{\mu}<\mu^2/L$, then the iteration complexity of GD is $\mu/\sqrt{L\hat{\mu}}\ge1$ times better. \thmref{theorem:acceleration_lower_bound} further strengthen this statement, by indicating that if our estimate does not depend on the true strong convexity parameter, then the iteration complexity of GD is even more favorable with a factor of $\mu/\hat{\mu}\ge1$, compared to our algorithm.

The proof of \thmref{theorem:acceleration_lower_bound}, which appears in Appendix \ref{proof:acceleration_lower_bound}, is again based on a reduction to an approximation problem via polynomials. In contrast to the proof of Theorem \ref{theorem:main_lower_bound} which employs Chebyshev polynomials, here only elementary algebraic manipulations are needed.

Another implication of \thmref{theorem:acceleration_lower_bound} is that the coefficients of optimal stationary $p$-CLIs for smooth and strongly convex functions must have an explicit dependence on the strong convexity parameter. In the next section we shall see that this fact is also responsible for the inability of stationary $p$-CLIs to obtain a rate of $\cO(\sqrt{L/\epsilon})$ for $L$-smooth convex functions.

									\subsection{No Acceleration for Stationary Algorithms over Smooth Convex Functions}
											\label{subsection:acceleration_and_stationarity}
Below, we prove that, as opposed to oblivious $p$-CLIs, stationary $p$-CLIs (namely, $p$-CLIs with time-invariant coefficients) over $L$-smooth convex functions can obtain an iteration complexity no better than $\cO(L/\epsilon)$. An interesting implication of this is that some current methods for minimizing finite sums of functions, such as SAG and SAGA (which are in fact stationary $p$-CLIs) cannot be optimal in this setting, and that time-changing coefficients are essential to get optimal rates. This further motivates the use of current acceleration schemes (e.g., \cite{frostig2015regularizing,lin2015universal}) which turn a given stationary algorithm into an non-stationary oblivious one.

The proof of this result is based on a reduction from the class of $p$-CLIs over $L$-smooth convex functions to $p$-CLIs over $L$-smooth and $\mu$-strongly convex, where the strong convexity parameter is given explicitly. This reduction allows us to apply the lower bound in \thmref{theorem:acceleration_lower_bound} on $p$-CLIs designed for smooth non-strongly convex functions.

We now turn to describe the reduction in detail. In his seminal paper, Nesterov \cite{nesterov1983method} presents the AGD algorithm and shows that it obtains a convergence rate of
\begin{align}
	f(\bx^k)-f(\bx^*)\le \frac{4L\norm{\bx^0-\bx^*}^2}{(k+2)^2}
\end{align}
for $L$-smooth convex functions, which admits at least one minimizer (accordingly, throughout the rest of this section we shall assume that the functions under consideration admit at least one minimizer, i.e., $\argmin(f)\neq\emptyset$). In addition, Nesterov proposes a restarting scheme of this algorithm which, assuming the strong convexity parameter is known, allows one to obtain an iteration complexity of $\bigtO{\sqrt{\kappa}\ln (1/\epsilon)}$. Scheme \ref{table:restart_scheme} shown below forms a simple generalization of the scheme discussed in that paper, and allows one to explicitly introduce a strong convexity parameter into the dynamics of (not necessarily oblivious) $p$-CLIs over $L$-smooth convex functions. 

\begin{table}[ht]
\vskip 0.15in
\begin{center}
\begin{small}
\begin{sc}
\begin{tabular}{llccr}\label{table:restart_scheme}
\small \textbf{Scheme 4.2} &\small Restarting Scheme \\
\hline
\scriptsize \textbf{Parameters} &\scriptsize  $\centerdot$ Smoothness parameter $L>0$\\
 &\scriptsize  $\centerdot$ Strong convexity parameter $\mu>0$\\
&\scriptsize  $\centerdot$ Convergence parameters $\alpha>0,C>0$\\
\scriptsize \textbf{Given}& \scriptsize A $p$-CLI over $L$-smooth functions $\cP$ with \\
&\scriptsize  \quad $f(\bx^k)-f^*\le \frac{CL\norm{\bar{\bx}^0-\bx^*}^2}{k^\alpha}$\\
&\scriptsize \quad for any initialization vector $\bar{\bx}^0$\\
\scriptsize \textbf{Iterate} &\scriptsize for $t=1,2,\dots$\\
&\scriptsize \quad Restart the step size schedule of $\cP$ \\
&\scriptsize \quad Initialize $\cP$ at $\bar{\bx}^0$ \\
&\scriptsize \quad Run $\cP$  for $\sqrt[\alpha]{4CL/\mu}$ iterations\\
&\scriptsize \quad Set $\bar{\bx}^0$ to be the last iterate of this execution\\
\scriptsize \textbf{End}\\
\hline
\end{tabular}
\end{sc}
\end{small}
\end{center}
\vskip -0.1in

\end{table}
The following lemma provides an upper bound on the iteration complexity of $p$-CLIs obtained through Scheme \ref{table:restart_scheme}.

\begin{lemma}\label{lemma:restart_scheme}
The convergence rate of a $p$-CLI algorithm obtained by applying Scheme \ref{table:restart_scheme}, using the corresponding set of parameters $L,\mu,C,\alpha$, is 
\begin{align*}
	\bigtO{\sqrt[\alpha]{{\kappa}}\ln(1/\epsilon)},
\end{align*} 
where $\kappa=L/\mu$ denotes the condition number.
\end{lemma} 
\begin{proof}
Suppose $\cP$ is a $p$-CLI as stated in Scheme \ref{table:restart_scheme} and let $f$ be a $L$-smooth and $\mu$-strongly convex function.
Each external iteration in this scheme involves running $\cP$ for $k=\sqrt[\alpha]{4CL/\mu}$ iterations, Thus, for any arbitrary point $\bar{\bx}$,
\begin{align*}
	f(\bx^{(k)})-f^*&\le \frac{CL\norm{\bar{\bx}-\bx^*}^2}{(\sqrt[\alpha]{4CL/\mu})^\alpha}=\frac{\norm{\bar{\bx}-\bx^*}^2}{4/\mu}.
\end{align*}
Also, $f$ is $\mu$-strongly convex, therefore
\begin{align*}
	f(\bx^{(k)})-f^*&\le \frac{ 2(f(\bar{\bx})-f(\bx^*))/\mu   } {4/\mu}
	\le \frac{ f(\bar{\bx})-f(\bx^*)   } {2}.
\end{align*}
That is, after each external iteration the sub-optimality in the objective value is halved. Thus, after $T$ external iterations, we get
\begin{align*}
	f(\bx^{(T\sqrt[\alpha]{4CL/\mu})})-f^*&\le \frac{ f(\bar{\bx}^0)-f(\bx^*)   } {2^T},
\end{align*}
where $\bar{\bx}^0$ denotes some initialization point. Hence, the iteration complexity for obtaining an $\epsilon$-optimal solution is
\begin{align*}
	\sqrt[\alpha]{4C\kappa}\log_2\circpar{ \frac{f(\bar{\bx}^0)-f(\bx^*)  }{\epsilon}}.
\end{align*}
\end{proof}											


The stage is now set to prove the statement made at the beginning of this section. Let $\cP$ be a stationary $p$-CLI over $L$-smooth functions with a convergence rate of $\bigO{L/k^\alpha}$, and let $\mu\in (0,L)$ be the strong convexity parameter of the function to be optimized. We apply Scheme \ref{table:restart_scheme} to obtain a new $p$-CLI, which according to \lemref{lemma:restart_scheme}, admits an iteration complexity of  $\bigO{\sqrt[\alpha]{\kappa}\ln(1/\epsilon)}$. But, since $\cP$ is stationary, the resulting $p$-CLI under Scheme \ref{table:restart_scheme} is again $\cP$ (That is, stationary $p$-CLIs are invariant w.r.t. Scheme \ref{table:restart_scheme}). Now, $\cP$ is a $p$-CLI over smooth non-strongly convex, and as such, its coefficients do not depend on $\mu$. Therefore, by \thmref{theorem:acceleration_lower_bound}, we get that $\alpha\le1$. Thus, we arrive at the following corollary:
\begin{corollary}
If the iteration complexity of a given stationary $p$-CLI over $L$-smooth functions is 
\begin{align*}
\bigO{\sqrt[\alpha]{L/\epsilon}},
\end{align*} 
then $\alpha\le1$.
\end{corollary}

The lower bound above is tight and is attained by standard Gradient Descent. 


																	\section{Summary}
In this work, we propose the framework of first-order $p$-CLIs and show that it can be efficiently utilized to derive bounds on the iteration complexity of a wide class of optimization algorithms, namely, oblivious, possibly stochastic, $p$-CLIs over smooth and strongly-convex functions. 

We believe that these results are just the tip of the iceberg, and the generality offered by this framework can be successfully instantiated for many other classes of algorithms. For example, it is straightforward to derive a lower bound of $\Omega(1/\epsilon)$ for 1-CLIs over 1-Lipschitz (possibly non-smooth) convex functions using the following set of functions
\begin{align*}
	\left\{ \|\bx-\bc\|\middle| \bc\in \reals^d \right\}.
\end{align*}
How to derive a lower bound for other types of $p$-CLIs in the non-smooth setting is left to future work.

\paragraph{Acknowledgments:}
This research is supported in part by an FP7 Marie Curie CIG grant, the Intel ICRI-CI Institute, and Israel
Science Foundation grant 425/13. We thank Nati Srebro for several helpful discussions and insights.

\bibliography{../../../mybib}
\bibliographystyle{plain}

\newpage
\appendix
\onecolumn

\section{Proofs}
\subsection{Proof for \thmref{theorem:main_lower_bound}} \label{proof:main_lower_bound}
Let us apply the given oblivious $p$-CLI algorithm on a quadratic function of the form $$f:\reals^d\to\reals:~\bx\mapsto \frac{1}{2}\bx^\top Q \bx + \bq^\top \bx,$$ where $Q=\diag{\eta,\dots,\eta}$ and $\bq=-\bv \eta$ for some $\eta\in[\mu,L]$ and $\bv\neq0\in\reals^d$. In particular, we have that the norm of the unique minimizer is $\norm{\bx^*}=\norm{-Q^{-1}\bq}=\norm{\bv}$.  We set the initialization points to be zero, i.e., $\bE \bx_j =0,~j=1,\dots,p$, and denote the corresponding coefficients by $A^{(k)}_{ij},B^{(k)}_{ij}\in \reals^{d\times d}$. The crux of proof is that, as long as $\eta$ lies in $[\mu,L]$, the side-information $\{\mu,L\}$ remains consistent, and therefore, the coefficients remain unchanged. \\

First, we express $\bE \bx_i^{k+1}$ in terms of $Q,\bq$ and $\bE \bx^{(k)}_{1},\dots,\bE \bx^{(k)}_{p} \in \reals^d$. By \defref{definition:iteration_complexity} we have for any $i\in[p]$,
\begin{align*}
	\bE\bx^{k+1}_i &= \sum_{j=1}^{p} \circpar{A^{(k)}_{ij} \partial f  + B^{(k)}_{ij} }(\bE\bx^{(k)}_{j})\\
	&= \sum_{j=1}^{p} (A^{(k)}_{ij}\partial f(\bE\bx^{(k)}_j) +B^{(k)}_{ij}\bE\bx^{(k)}_j)\\
	&= \sum_{j=1}^{p} (A^{(k)}_{ij}(Q\bE\bx^{(k)}_j+ \bq)+B_{ij}\bE\bx^{(k)}_j)\\
	&= \sum_{j=1}^{p} (A^{(k)}_{ij}Q + B^{(k)}_{ij})\bE\bx^{(k)}_j + \sum_{j=1}^{p} A^{(k)}_{ij} \bq.
\end{align*}
Our next step is to reduce the problem of minimizing $f$ to a polynomial approximation problem. 
We claim that 
for any $k\ge1$ and $i\in[d]$  there exist $d$ real polynomials $s_{k,i,1}(\eta),\dots,s_{k,i,d}(\eta)$ of degree at most $k-1$, such that 
\begin{align} \label{eq:induction_claim1}
	\bE \bx_i^{(k)}   = (s_{k,i,*}(\eta))\eta,
\end{align}
where $$(s_{k,i,*}(\eta)) \coloneqq (s_{k,i,1}(\eta),\dots,s_{k,i,d}(\eta))^\top.$$
Let us prove this claim using mathematical induction. For $k=1$ we have
\begin{align} \label{eq:induction_basis_1}
	\bE\bx^{(1)}_i 	&= \sum_{j=1}^{p} (A^{(0)}_{ij}Q + B^{(0)}_{ij})\bE\bx^{(0)}_j + \sum_{j=1}^{p} A^{(0)}_{ij} \bq 
	 	= -\sum_{j=1}^{p} A^{(0)}_{ij} \bv\eta,
\end{align}
showing that the base case holds. For the induction step, assume the statement 
holds for some $k>1$ with $s_{k,i,j}(\eta)$ as above. Then,
\begin{align} \label{eq:induction_step_1}
	\bE\bx^{(k+1)}_i 	&= \sum_{j=1}^{p} (A^{(k)}_{ij}Q + B^{(k)}_{ij})\bE\bx^{(k)}_j + \sum_{j=1}^{p} A^{(k)}_{ij} \bq \nonumber\\ 
	&= \sum_{j=1}^{p} (A^{(k)}_{ij}\diag{\eta,\dots,\eta} + B^{(k)}_{ij}) (s_{k,j,*}(\eta))\eta - \sum_{j=1}^{p} A^{(k)}_{ij} \bv\eta \nonumber\\
	&= \left(\sum_{j=1}^{p} (A^{(k)}_{ij}\diag{\eta,\dots,\eta} + B^{(k)}_{ij}) (s_{k,j,*}(\eta)) 
	-\sum_{j=1}^{p} A^{(k)}_{ij} \bv\right) \eta.
\end{align}
The expression inside the last parenthesis is a vector with $d$  entries, each of which contains a real polynomial of degree at most $k$. This concludes the induction step (note that the derivations of equalities (\ref{eq:induction_basis_1}) and (\ref{eq:induction_step_1}) above are exactly where we use the fact that there is no functional dependency of $A^{(k)}_{ij}$ and $B^{(k)}_{ij}$ on $\eta$).

We are now ready to estimate the sub-optimality of $\bE\bx^{(k)}_p$, the expected point 
returned by the algorithm at the $k$'th iteration. Setting $m\in\argmax_j |\bv_j|$ we have
\begin{align} \label{ineq:argument_sub_optimal}
	\norm{\bE\bx_p^{(k)}-\bx^*}&= \| (s_{k,p,*}(\eta))\eta + \bv\|\ge | s_{k,p,m}(\eta)\eta + v_m|
	= |v_m||s_{k,p,m}(\eta)\eta/v_m + 1|.
\end{align}
By \lemref{lemma:smooth_and_strongly_convex_polynomials} in appendix B, there exists $\eta\in[\mu,L]$, such that 
\begin{align*}
	|s_{k,p,m}(\eta)\eta/v_m + 1|\ge  \circpar{\frac{\sqrt{\kappa}-1}{\sqrt{\kappa}+1}}^{k},
\end{align*}
where $\kappa=L/\mu$. Defining $Q$ and $\bq$ accordingly, and choosing, e.g., $\bv=R\be_1$ where $R$ denotes a prescribed distance, yields 
\begin{align*}
	\norm{\bE\bx_p^{(k)}-\bx^*}&\ge R \circpar{\frac{\sqrt{\kappa}-1}{\sqrt{\kappa}+1}}^{k}.
\end{align*}
Using the fact that $f$ is $\mu$-strongly convex concludes the proof of the first part of the theorem. 

For the smooth case we need to estimate $f(\bE\bx_p^{(k)})-f^*$. Let $\eta\in(0,L]$ and 
define $Q$ and $\bq$, accordingly. \ineqref{ineq:argument_sub_optimal} yields
\begin{align} \label{quadratic_function_for_smooth_case}
	 f(\bE\bx^{(k)}) - f(\bx^*) &=  \frac{1}{2}(\bE\bx^{(k)} - \bx^*)^\top Q(\bE\bx^{(k)} - \bx^*)\\
	&\ge  \frac{(v_m)^2}{2}\eta ((s_{k,p,m})(\eta)\eta/v_m + 1)^2. \nonumber
\end{align}
Now, by \lemref{lemma:smooth_polynomials} in appendix B,
\begin{align*}
	\min_{s(\eta),~\partial s \le k-1} \max_{\eta\in(0,L]}\eta(s(\eta)\eta + 1)^2 \ge \frac{L}{(2k+1)^2}
\end{align*}
Thus, choosing $\bv=R\be_1$ concludes the proof.

\subsection{Proof for \thmref{theorem:acceleration_lower_bound}} \label{proof:acceleration_lower_bound}

The proof of this theorem follows the exact reduction used in the proof of \thmref{theorem:main_lower_bound} (see Appendix \ref{proof:main_lower_bound} above). The only difference is that here $\mu$ is allowed to be any real number in $(0,L)$. This consideration reduces our problem into, yet another,  polynomial approximation problem. For completeness, we provide here the full proof.

Let us apply the given oblivious $p$-CLI algorithm on a quadratic function of the form $$f:\reals^d\to\reals:~\bx\mapsto \frac{1}{2}\bx^\top Q \bx + \bq^\top \bx,$$ where $Q=\diag{\eta,\dots,\eta}$ and $\bq=-\bv \eta$ for some $\eta\in(0,L)$ and $\bv\neq0\in\reals^d$. In particular, we have that the norm of the unique minimizer is $\norm{\bx^*}=\norm{-Q^{-1}\bq}=\norm{\bv}$.  We set the initialization points to be zero, i.e., $\bE \bx_j =0,~j=1,\dots,p$, and denote the corresponding coefficients by $A^{(k)}_{ij},B^{(k)}_{ij}\in \reals^{d\times d}$. The crux of proof is that, as long as $\eta$ lies in $(0,L]$, the side-information $\{\mu,L\}$ remains consistent, and therefore, the coefficients remain unchanged. \\

First, we express $\bE \bx_i^{k+1}$ in terms of $Q,\bq$ and $\bE \bx^{(k)}_{1},\dots,\bE \bx^{(k)}_{p} \in \reals^d$. By \defref{definition:iteration_complexity} we have for any $i\in[p]$,
\begin{align*}
	\bE\bx^{k+1}_i &= \sum_{j=1}^{p} \circpar{A^{(k)}_{ij} \partial f  + B^{(k)}_{ij} }(\bE\bx^{(k)}_{j})\\
	&= \sum_{j=1}^{p} (A^{(k)}_{ij}\partial f(\bE\bx^{(k)}_j) +B^{(k)}_{ij}\bE\bx^{(k)}_j)\\
	&= \sum_{j=1}^{p} (A^{(k)}_{ij}(Q\bE\bx^{(k)}_j+ \bq)+B_{ij}\bE\bx^{(k)}_j)\\
	&= \sum_{j=1}^{p} (A^{(k)}_{ij}Q + B^{(k)}_{ij})\bE\bx^{(k)}_j + \sum_{j=1}^{p} A^{(k)}_{ij} \bq 
\end{align*}
Our next step is to reduce the problem of minimizing $f$ to a polynomial approximation problem. 
We claim that 
for any $k\ge1$ and $i\in[d]$  there exist $d$ real polynomials $s_{k,i,1}(\eta),\dots,s_{k,i,d}(\eta)$ of degree at most $k-1$, such that 
\begin{align} \label{eq:induction_claim_2}
	\bE \bx_i^{(k)}   = (s_{k,i,*}(\eta))\eta,
\end{align}
where $$(s_{k,i,*}(\eta)) \coloneqq (s_{k,i,1}(\eta),\dots,s_{k,i,d}(\eta))^\top.$$
Let us prove this claim using mathematical induction. For $k=1$ we have,
\begin{align} \label{eq:induction_basis_2}
	\bE\bx^{(1)}_i 	&= \sum_{j=1}^{p} (A^{(0)}_{ij}Q + B^{(0)}_{ij})\bE\bx^{(0)}_j + \sum_{j=1}^{p} A^{(0)}_{ij} \bq = -\sum_{j=1}^{p} A^{(0)}_{ij} \bv\eta,
\end{align}
showing that the base case holds. For the induction step, assume the statement 
holds for some $k>1$ with $s_{k,i,j}(\eta)$ as above, then
\begin{align} \label{eq:induction_step_2}
	\bE\bx^{(k+1)}_i 	&= \sum_{j=1}^{p} (A^{(k)}_{ij}Q + B^{(k)}_{ij})\bE\bx^{(k)}_j + \sum_{j=1}^{p} A^{(k)}_{ij} \bq \nonumber\\ 
	&= \sum_{j=1}^{p} (A^{(k)}_{ij}\diag{\eta,\dots,\eta} + B^{(k)}_{ij}) (s_{k,j,*}(\eta))\eta \nonumber\\
	&- \sum_{j=1}^{p} A^{(k)}_{ij} \bv\eta \nonumber\\
	&= \left(\sum_{j=1}^{p} (A^{(k)}_{ij}\diag{\eta,\dots,\eta} + B^{(k)}_{ij}) (s_{k,j,*}(\eta)) \right.\nonumber\\
	&- \left.\sum_{j=1}^{p} A^{(k)}_{ij} \bv\right) \eta.
\end{align}
The expression inside the last parenthesis is a vector with $d$  entries, each of which contains a real polynomial of degree at most $k$. This concludes the induction step. (note that the derivations of equalities (\ref{eq:induction_basis_2}) and (\ref{eq:induction_step_2}) above are exactly where we use the fact that there is no functional dependency of $A^{(k)}_{ij}$ and $B^{(k)}_{ij}$ on $\eta$).

We are now ready to estimate the sub-optimality of $\bE\bx^{(k)}_p$, the point 
returned by the algorithm at the $k$'th iteration. Let us set $m\in\argmax_j |\bv_j|$, then 
\begin{align} \label{ineq:argument_sub_optimal_2}
	\norm{\bE\bx_p^{(k)}-\bx^*}&= \| (s_{k,p,*}(\eta))\eta + \bv\|\nonumber\\
	&\ge | s_{k,p,m}(\eta)\eta + v_m|\nonumber\\
	&= |v_m||s_{k,p,m}(\eta)\eta/v_m + 1|
\end{align}
By \lemref{lemma:smooth_and_unknown_strongly_convex_polynomials}, there exists $\eta\in(L/2,L)$, such that 
\begin{align} \label{ineq:accelerated_polynomial}
	|s_{k,p,m}(\eta)\eta/v_m + 1|\ge (1-\eta/L)^{k+1}.
\end{align}
Defining $Q$ and $\bq$ accordingly, and choosing, e.g., $\bv=R\be_1$ where $R$ denotes a prescribed distance, yields 
\begin{align*}
	\norm{\bE\bx_p^{(k)}-\bx^*}&\ge  (1-\eta/L)^{k+1}.
\end{align*}
Using the fact that $f$ is $L/2$-strongly convex concludes the proof.

																											\section{Technical Lemmas}

Below, we provide 3 lemmas which are used to bound from below the quantity $|s(\eta)\eta+1|$ over different domains of $\eta$, where $s(\eta)$ is a real polynomial. For brevity, we denote the set of real polynomials of degree $k$ by $\cP_k$.

\begin{lemma}\label{lemma:smooth_and_strongly_convex_polynomials}
Let $s(\eta)\in\cP_k$, and let $0<\mu<L$. Then,
\begin{align*}
	\max_{\eta\in[\mu,L]}|s(\eta)\eta+1|\ge \circpar{\frac{\sqrt{\kappa}-1}{\sqrt{\kappa}+1}}^{k+1}
\end{align*}
where $\kappa\coloneqq L/\mu$.
\end{lemma}
\begin{proof}
Denote $q(\eta)\coloneqq T_{k+1}^{-1}\circpar{\frac{L+\mu}{L-\mu}}T_{k+1}\circpar{\frac{2\eta-\mu-L}{L-\mu}}$,
where $T_k(\eta)$ denotes the Chebyshev polynomial of degree $k$,
\begin{align}\label{eq:Chebyshev}
	T_k(\eta)=
	\begin{cases}
		\cos(k\arccos(\eta))& |\eta|\le1\\
		\cosh(k\operatorname{arcosh}(\eta))& \eta\ge1\\
		(-1)^n\cosh(k\operatorname{arcosh}(-\eta))& \eta\le-1.
	\end{cases}
\end{align}
It follows that $|T_{k+1}(\eta)|\le1$ for $\eta\in[-1,1]$ and $$T_{k+1}(\cos(j\pi/(k+1)))=(-1)^j,~j=0,\dots,k+1.$$ Accordingly, $|q(\eta)|\le T_{k+1}^{-1}\circpar{\frac{L+\mu}{L-\mu}},~\eta\in[\mu,L]$ and 
\begin{align*}
q\circpar{ \theta_j }=(-1)^jT_{k+1}^{-1}\circpar{\frac{L+\mu}{L-\mu}},~j=0,\dots,k+1,
\end{align*}
where
\begin{align*}
	\theta_j=\frac{\cos(j\pi/(k+1))(L-\mu) + \mu + L}{2}. 
\end{align*}
Suppose, for the sake of contradiction, that 
\begin{align*}
\max_{\eta\in[\mu,L]}|s(\eta)\eta+1|<\max_{\eta\in[\mu,L]}|q(\eta)|.
\end{align*}
Thus, for $r(\eta)=q(\eta)-(1+s(\eta)\eta))$, we have $r(\theta_j)>0$  for even $j$, and $r(\theta_j)<0$ for odd $j$. Hence, $r(\eta)$ has $k+1$ roots in $[\mu,L]$. But, since $r(0)=0$ and $\mu>0$, it follows $r(\eta)$ has at least $k+2$ roots, which contradicts the fact that the degree of $r(\eta)$ is at most $k+1$. Therefore,
\begin{align*}
\max_{\eta\in[\mu,L]}|s(\eta)\eta+1|\ge\max_{\eta\in[\mu,L]}|q(\eta)|= T_{k+1}^{-1}\circpar{\frac{\kappa+1}{\kappa-1}},
\end{align*}
where $\kappa=L/\mu$. Since $(\kappa+1)/(\kappa-1)\ge1$, we have by \eqref{eq:Chebyshev},
\begin{align*}
T_k\circpar{\frac{\kappa+1}{\kappa-1}}&=\cosh\circpar{k\operatorname{arcosh}\circpar{\frac{\kappa+1}{\kappa-1}}}\\
&=\cosh\circpar{k \ln\circpar{\frac{\kappa+1}{\kappa-1}+\sqrt{\circpar{\frac{\kappa+1}{\kappa-1}}^2-1} }}\\
&=\cosh\circpar{k \ln\circpar{\frac{\kappa+2\sqrt\kappa+1}{\kappa-1} }}\\
&=\cosh\circpar{k \ln\circpar{\frac{ \sqrt\kappa+1}{\sqrt\kappa-1} }}\\
&= \frac{1}{2}\circpar{\circpar{\frac{ \sqrt\kappa+1}{\sqrt\kappa-1} }^k+\circpar{\frac{\sqrt\kappa-1}{ \sqrt\kappa+1} }^k  }\\
&\le \circpar{\frac{ \sqrt\kappa+1}{\sqrt\kappa-1}}^k.
\end{align*}
Hence,
\begin{align*}
	\max_{\eta\in[\mu,L]}|s(\eta)\eta+1|\ge \circpar{\frac{ \sqrt\kappa-1}{\sqrt\kappa+1}}^{k+1}
\end{align*}
\end{proof}

\begin{lemma}\label{lemma:smooth_polynomials}
Let $s(\eta)\in\cP_k$, and let $0<L$. Then,
\begin{align*}
	\max_{\eta\in[0,L]}\eta|s(\eta)\eta+1|^2\ge \frac{L}{(2k+3)^2}
\end{align*}
\end{lemma}
\begin{proof}
First, we define 
\begin{align*}
	q(\eta)&= 
	\begin{cases}
			(-1)^k(2k+3)^{-1}\sqrt{L/\eta} T_{2k+3}(\sqrt{\eta/L}) &\eta\neq0\\0&\eta=0
	\end{cases}
\end{align*}
where $T_k(\eta)$ is the $k$'th Chebyshev polynomial (see (\ref{eq:Chebyshev})). Let us show that $q(\eta)$ is a polynomial of degree $k+1$ and that $q(0)=1$. The following trigonometric identity
\begin{align*}
	\cos\alpha  +\cos\beta=2\cos\circpar{\frac{\alpha-\beta}{2}}\cos\circpar{\frac{\alpha+\beta}{2}},
\end{align*}
together with (\ref{eq:Chebyshev}), yields the following recurrence formula
\begin{align*}
	T_k(\eta)= 2\eta T_{k-1}(\eta) - T_{k-2}(\eta).
\end{align*}
Noticing that $T_0(\eta)=1$ and $T_1(\eta)=x$ (also by (\ref{eq:Chebyshev})), we can use mathematical induction to prove that Chebyshev polynomials of odd degree have only odd powers and that the corresponding coefficient for the first power $\eta$ in $T_{2k+3}(\eta)$ is indeed $(-1)^k(2k+3)$. Equivalently, we get that $q(\eta)$ is a polynomial of degree $k+1$ and that $q(0)=1$. Next, note that for 
\begin{align*}
	\theta_j&= L\cos\circpar{\frac{j\pi}{2k+3}}^2\in[0,L],\quad j=0,\dots,k+1
\end{align*}
we have
\begin{align*}
	\max_{\eta\in[0,L]}\eta^{1/2}|q(\eta)|=(-1)^j\theta_j^{1/2} q(\theta_j) = \frac{\sqrt{L}}{2k+3}.
\end{align*}
Now, suppose, for the sake of contradiction, that
\begin{align*}
	 \max_{\eta\in[0,L]}\eta|s(\eta)\eta+1|^2< \max_{\eta\in[0,L]}\eta|q(\eta)|^2.
\end{align*}
In particular,
\begin{align*}
	 \theta_j^{1/2}|s(\theta_j)\theta_j^{1/2}+1|< \theta_j^{1/2}|q(\theta_j)|.
\end{align*}
Since $\theta_j>0$, we have
\begin{align*}
	 |s(\theta_j)\theta_j^{1/2}+1|< |q(\theta_j)|.
\end{align*}
We proceed in a similar way to the proof of \lemref{lemma:smooth_and_strongly_convex_polynomials}. For $r(\eta)=q(\eta)-(1+s(\eta)\eta))$, we have $r(\theta_j)>0$  for even $j$, and $r(\theta_j)<0$ for odd $j$. Hence, $r(\eta)$ has $k+1$ roots in $[\theta_{k+1},L]$. But, since $r(0)=0$ and $\theta_{k+1}>0$, it follows $r(\eta)$ has at least $k+2$ roots, which contradicts the fact that degree of $r(\eta)$ is a at most $k+1$. Therefore,
\begin{align*}
	 \max_{\eta\in[0,L]}\eta|s(\eta)\eta+1|^2\ge \max_{\eta\in[0,L]}\eta|q(\eta)|^2\ge\frac{L}{(2k+3)^2}
\end{align*}
concluding the proof.

\end{proof}

\begin{lemma}\label{lemma:smooth_and_unknown_strongly_convex_polynomials}
Let $s(\eta)\in\cP_k$, and let $0<L$. Then exactly one of the two following holds:
\begin{enumerate}
	\item For any $\epsilon>0$, there exists $\eta\in(L-\epsilon,L)$ such that 
\begin{align*}
	|s(\eta)\eta+1|>(1-\eta/L)^{k+1}.
\end{align*}
	\item $s(\eta)\eta+1=(1-\eta/L)^{k+1}$.
\end{enumerate}	
\end{lemma}

\begin{proof}
It suffices to show that if (1) does not hold then $s(\eta)\eta+1=(1-\eta/L)^{k+1}$. Suppose that there exists $\epsilon>0$ such that for all $\eta\in(L-\epsilon,L)$ it holds that 
\begin{align*}
	|s(\eta)\eta+1| \le  \left(1-\frac{\eta}{L} \right)^{k+1}.
\end{align*}
Define 
\begin{align} \label{def:q_polynomial}
q(\eta)\coloneqq s\left(L(1-\eta)\right)L(1-\eta) + 1
\end{align}
and denote the corresponding coefficients by $q(\eta)=\sum_{j=0}^{k+1} q_i \eta^j$.
We show by induction that $q_j=0$ for all $j=0,\dots,k$. \\ For $j=0$ we have that since for any $\eta\in(0, 1- (L-\epsilon)/L  )$
\begin{align*}
	|q(\eta)|\le \left(1-\frac{L(1-\eta)}{L} \right)^{k+1} = \eta^{k+1},
\end{align*}
it holds that 
\begin{align*}
	|q_0|=|q(0)|=\left|\lim_{\eta\to 0^+} q(\eta)\right|&\le \lim_{\eta\to 0^+} \eta^{k+1}=0.
\end{align*}
Now, if $q_0=\dots=q_{m-1}=0$ for $m<k+1$ then 
\begin{align*}
	|q_{m}|=\left|\frac{q(0)}{\eta^{m}}\right|=\left|\lim_{\eta\to 0^+} \frac{q(\eta)}{\eta^{m}}\right|&\le \lim_{t\to 0^+} \eta^{k+1-m}=0.
\end{align*}
Thus, proving the induction claim. This, in turns, implies that $q(\eta)=q_{k+1}\eta^{k+1}$. Now, by \eqref{def:q_polynomial}, it follows that $q_{k+1}=q(1)=1$. Hence, $q(\eta)=\eta^{k+1}$. Lastly, using \eqref{def:q_polynomial} again yields
\begin{align*}
	s(\eta)\eta+1 = q\left(1-\frac{\eta}{L}\right) =\left(1-\frac{\eta}{L}\right)^{k+1},
\end{align*}
concluding the proof.
\end{proof}

\end{document}